\documentclass[a4paper, 12pt]{article}
\usepackage{tgheros}
\usepackage{amsthm}
\newtheorem{thm}{Theorem}
\newtheorem{cor}[thm]{Corollary}

\newtheorem{lemma}[thm]{Lemma}

\usepackage{amssymb,amsmath}

\newtheorem*{theorem*}{Theorem}

\DeclareMathOperator{\F}{\mathbb{F}}

\begin{document}
\baselineskip=16.3pt
\parskip=14pt

\begin{center}
\section*{On Galois groups of linearized polynomials related to the special linear group of prime degree}

{\large 
Rod Gow and Gary McGuire
 \\ { \ }\\
School of Mathematics and Statistics\\
University College Dublin\\
Ireland}
\end{center}

 \subsection*{Abstract}
 
Let $F$ be any field of prime characteristic $p$ and let $q$ be a power of $p$. 
We assume that $F$ contains the finite field
of order $q$. A $q$-polynomial $L$ over $F$ is an element of the polynomial ring $F[x]$ with the property that
all exponents are a power of $q$. 
We assume that the coefficient of the $x$ term of $L$ is nonzero. We investigate the Galois group $G$
of $L$ over $F$, under the assumption that $L(x)/x$ is irreducible in $F[x]$. It is well known that if $L$ has $q$-degree $n$, its roots are
an $n$-dimensional vector space over the field of order $q$ and $G$ acts linearly on this space.

Our main theorem is the following. We consider a monic $q$-polynomial $L$ whose $q$-degree is a prime, $r$, say, with $L(x)/x$ irreducible over $F$. 
Assuming that the coefficient of the $x$ term of $L$ is $(-1)^r$, 
we show that the Galois group of $L$ over $F$
must be the special linear group $SL(r,q)$ of degree $r$ over the field of order $q$ when $q>2$.
We also prove a projective version for the projective special linear group $PSL(r,q)$, and we present
the analysis when $q=2$.
  
  \vskip1in
  
  {\bf Keywords} Galois group, linearized polynomial, special linear group

  \newpage
  
\section{Introduction}
  
Let $p$ be a prime number, and let $q$ be a positive power of $p$.
Let $\F_q$ denote the finite field with $q$ elements.
Let $F$ be a field of characteristic $p$,
and assume that $F$ contains $\F_q$.
A $q$-linearized polynomial over $F$  is a polynomial of the form 
\begin{equation}\label{lin1}
L=a_0x+a_1x^q+a_2x^{q^2}+\cdots+a_n x^{q^n} \in F[x].
\end{equation}
If $a_n\not=0$ we say that $n$ is the $q$-degree of $f$.
We will usually say $q$-polynomial instead of $q$-linearized polynomial.

The set of roots of a $q$-polynomial $L$ forms an $\F_q$-vector space,
which is contained in a splitting field of $L$. We make this statement more precise in the following lemma.

\begin{lemma} \label{vector_space}
Let $F$ be a field of prime characteristic $p$ that contains $\F_q$. Let $L$ be a $q$-polynomial of $q$-degree $n$ in $F[x]$, with
\[
L=a_0x+a_1x^q+a_2x^{q^2}+\cdots+a_n x^{q^n}.
\]
Let $E$ be a splitting field for $L$ over $F$ and let $V$ be the set of roots of $L$ in $E$. Let $G$ be the Galois group of $E$
over $F$. Suppose that $a_0\neq 0$. Then $V$ is an $\F_q$-vector space of dimension $n$ and $G$ is naturally a subgroup
of $GL(n,q)$.
\end{lemma}

\begin{proof}
The derivative $L'(x)=a_0$ and is thus a nonzero constant under the hypothesis above. It follows that $L$ has no repeated
roots and thus $|V|=q^n$. Let $\alpha$ and $\beta$ be elements of $V$. Then $L(\alpha)=L(\beta)=0$. Since
\[
(\alpha+\beta)^{q^i}=\alpha^{q^i}+\beta^{q^i}
\]
for all $i\geq 0$, it is clear that $L(\alpha+\beta)=0$, and hence $\alpha+\beta\in V$.

Let $\lambda$ be an element of $\F_q$. Since $\lambda^q=\lambda$, we have $(\lambda\alpha)^{q^i}=\lambda \alpha^{q^i}$. It
follows that $L(\lambda\alpha)=\lambda L(\alpha)=0$ and hence $\lambda\alpha\in V$. These arguments show that $V$ is a vector space
over $\F_q$. Furthermore, since $|V|=q^n$, $V$ has dimension $n$ over $\F_q$.

Let $\sigma$ be an element of $G$. By definition of Galois group action, $G$ permutes the roots of $L$ and hence maps $V$ into itself.
In addition, since $G$ is a group of field automorphisms of $E$, we have
\[
\sigma(\alpha+\beta)=\sigma(\alpha)+\sigma(\beta)
\]
for all $\alpha$ and $\beta$ in $V$. 

Now let $\lambda$ be an element of $\F_q$. Since $\F_q$ is assumed to be a subfield of $F$ and $G$ fixes $F$ elementwise, we have
$\sigma(\lambda)=\lambda$. Then, again by definition of Galois group action,
\[
\sigma(\lambda\alpha)=\sigma(\lambda)\sigma(\alpha)=\lambda\sigma(\alpha).
\]
This shows that the action of $G$ on $V$ is $\F_q$-linear and thus $G$ may be considered to be a subgroup of $GL(n,q)$.
\end{proof}

We call $V$ the ($\F_q$-vector) space of roots of $L$.

The next result is a trivial consequence of elementary Galois theory but it plays an important
role in the paper.

\begin{lemma} \label{transitive_action}
Let $F$ be a field of prime characteristic $p$ that contains $\F_q$. Let $L$ be a $q$-polynomial in $F[x]$, such
that $L(x)/x$ is irreducible over $F$. 
Let $E$ be a splitting field for $L$ over $F$ and let $V$ be the space of roots of $L$ in $E$. Let $G$ be the Galois group of $E$
over $F$. Then $G$ acts transitively on the nonzero elements of $V$.
\end{lemma}

\begin{proof}
$G$ permutes the roots of $L(x)/x$, which are the nonzero elements of $V$. Since $L(x)/x$ is irreducible, the action of $G$
on the roots of $L(x)/x$ is transitive.
\end{proof}

It follows from these two lemmas that
our Galois groups are subgroups 
of $GL(n,q)$ that act transitively on the nonzero
vectors of the underlying $n$-dimensional space.
There are a number of papers in the literature on this topic, 
as we will discuss  in Sections 3 and 4.
We use these results to prove our main theorem, as stated in the abstract, in Section 4.
Before that, we present some elementary calculations in Section 2.
Our main theorem is false for $q=2$, and we present 
the corresponding theorem when $q=2$ in Section 5.
In Section 6 we present a  theorem for $PSL(n,q)$,
and in Section 7 we discuss the case $n=2$.

We remark that $SL(n,q)$ has been realized as a Galois group over $\F_q (t)$ by several authors,
however our main result (Theorem  \ref{Galois_group_is_SL(r,q)}) has a more general ground field.

\section{Calculations with the determinant}

\noindent Let $F$ be a field of characteristic $p$ and $L$ be a monic $q$-polynomial of $q$-degree $n$ in $F[x]$, with nonzero coefficient of $x$. Let $\alpha_1$, \dots, $\alpha_n$
be an $\F_q$-basis for the space of roots of $L$ in a splitting field $E$ of $L$ over $F$. We define
the $(n+1)\times (n+1)$ matrix $A$ by
\[
A:=
 \left[ \begin{array}{ccccc}
   \alpha_1&\alpha_1^q&\alpha_1^{q^2}&\cdots&\alpha_1^{q^{n}}\\
 \alpha_2&\alpha_2^q&\alpha_2^{q^2}&\cdots&\alpha_2^{q^{n}}\\
 \vdots&\vdots&\vdots&\ddots&\vdots\\
  \alpha_n&\alpha_n^q&\alpha_n^{q^2}&\cdots&\alpha_n^{q^{n}}\\
  x&x^q&x^{q^2}&\cdots&x^{q^{n}}\\
 \end{array} \right].
\]
It is clear that $\det A$ is a $q$-polynomial with coefficients in $E$,
and we claim that its roots are the same as those of $L$. 
For if we set $x=\alpha_i$, the determinant of the resulting matrix is zero. 
Thus the $\alpha_i$ are 
roots of $\det A$, and since $\det A$ is a $q$-polynomial,
each element of the space of roots of $L$ is a root of $\det A$. Thus $\det A$ has $q^n$ different roots in $L$. As $\det A$
has $q$-degree at most $n$, we deduce that $\det A$ has $q$-degree exactly $n$ and hence the coefficient of $x^{q^n}$ is nonzero.

Let $D$ denote the $n\times n$ matrix obtained by omitting the last row and last column of $A$.
 It is clear by expansion
of $\det A$ that $\det D$ is the coefficient of $x^{q^n}$ in $\det A$. 
Therefore $\det D$ is a nonzero element of $E$. 
Let $\delta=\det D$.

\begin{lemma}\label{inF}
Continuing the notation above, 
\begin{enumerate}
\item $\delta^{q-1}\in F$, and
\item $F(\delta)$ is a normal extension of $F$.
\end{enumerate}
\end{lemma}

\begin{proof}
Since $\det A$ and $L$ have the same $q$-degree and same roots,
and since $L$ is assumed to be monic, and $\delta$ is the leading coefficient 
of $\det A$, we have
\[
L=\frac{1}{\delta} \det A.
\]
Expansion of $\det A$ shows that the coefficient of $x$ in $\det A$ is $(-1)^n \delta^q$. 
It follows that the coefficient
of $x$ in $L$ is $(-1)^n \delta^{q-1}$. An immediate consequence is that $\delta^{q-1}\in F$.

It follows that $\delta$ is a root in $E$ of the polynomial 
$x^{q-1}-\delta^{q-1}$, which has coefficients in $F$. 
Since the roots of $x^{q-1}-\delta^{q-1}$ are the elements 
$\omega\delta$, where $\omega$ ranges over the nonzero elements of $\F_q$,
$E$ contains the splitting field $F(\delta)$ of $x^{q-1}-\delta^{q-1}$ over $F$. 
Thus $F(\delta)$ is a normal subfield of $E$.
\end{proof}

Let $G$ be the Galois group of $E$ over $F$.
Since $F(\delta)$ is a normal extension of $F$, it is mapped into itself by $G$.
The next lemma tells us how $G$ acts on $F(\delta)$.

\begin{lemma}\label{sigmalemma}
Continuing the notation above, 
 let $\sigma$ be an element of $G$.
Then
\begin{equation}\label{sd1}
\sigma(\delta)=(\det \sigma) \ \delta.
\end{equation}
\end{lemma}

\begin{proof}
 Since the roots $\alpha_i$ of $L$, $1\leq i\leq n$, are an $\F_q$-basis of the space of roots of $F$, 
 the action of $\sigma$ on $E$ is determined by its matrix with respect to this basis.
We thus set $S_\sigma$ to be the matrix of $\sigma$ with respect to the $\alpha_i$ basis. A straightforward calculation using the
fact that $S_\sigma$ has entries in $\F_q$ shows that we have the matrix equation
\[
S_\sigma D=\sigma(D),
\]
where $\sigma(D)$ is obtained from $D$ by applying $\sigma$ to the entries of $D$. We now take determinants on each side
to obtain \eqref{sd1}.
\end{proof}

\begin{cor}
Let $L$ be a monic $q$-polynomial of $q$-degree $n$ in $F[x]$ and let $E$ be its splitting field over $F$.
Let $G$ be the Galois group of $L$, considered as a subgroup of $GL(n,q)$. 
Then $G \le SL(n,q)$ if and only if $\delta \in F$.
\end{cor}

\begin{proof}
If $G \le SL(n,q)$ then $\sigma (\delta)=\delta$ for all $\sigma \in G$ by Lemma  \ref{sigmalemma}, and conversely.
\end{proof}

  We can now obtain a nice consequence of this analysis.

\begin{cor} \label{determinants_are_1}
Let $L$ be a monic $q$-polynomial of $q$-degree $n$ in $F[x]$ and let $E$ be its splitting field over $F$. Let the coefficient of $x$ in $L$ be $(-1)^n$. Let
$G$ be the Galois group of $L$, considered as a subgroup of $GL(n,q)$. Then $\det \sigma=1$ for all $\sigma\in G$ and thus
$G$ is a subgroup of $SL(n,q)$.
\end{cor}

\begin{proof}
With the notation as before, $(-1)^n \delta^{q-1}$ is the coefficient 
of $x$ in $L$ and hence the hypotheses imply $\delta^{q-1}=1$. This implies that
$\delta \in \F_q$. 
Since we are assuming that $\F_q$ is contained in $F$, $\delta \in F$,
and so each element $\sigma$ of $G$ fixes $\delta$. Equation \eqref{sd1}
 implies that $\det \sigma=1$, as required.
\end{proof}

We note that the same conclusion  holds if the coefficient of $x$ is $(-1)^n \mu^{q-1}$, where $\mu$ is any nonzero
element of $F$.

\section{Transitive actions on vector spaces}

\noindent 
The results of the previous sections lead immediately to the following conclusion concerning a special type of Galois group
action. 

\begin{thm} \label{transitive_subgroup_of_SL}
Let $L$ be a monic $q$-polynomial of $q$-degree $n$ in $F[x]$. Suppose that $L(x)/x$ is irreducible over $F$ and the coefficient
of $x$ in $L$ is $(-1)^n$. Then the Galois group $G$ of $L$ over $F$ is isomorphic to a subgroup of the special linear
group $SL(n,q)$ that acts transitively on the nonzero vectors of $\mathbb{F}_q^n$.
\end{thm}

We turn to consideration of a group $G$ satisfying the conclusions of the theorem above when $n$ is a prime, $r$, say. The main result of this paper is that
if $r$ is an odd prime and $q\not=2$, then $G$ must be equal to $SL(r,q)$.

We next provide some background detail. Let $q$ be a power of $p$. 
For the moment we let $n$ be any positive integer, later we shall assume that $n=r$ a prime.
We shall choose the finite  field
$\F_{q^n}$ as our  $n$-dimensional vector space over $\mathbb{F}_q$ on which $GL(n,q)$ acts.
Let $\alpha$ be a generator of the multiplicative group $\F_{q^n}\setminus \{0\}$ and define 
$\tau:\F_{q^n}\to \F_{q^n}$ by $\tau(z)=\alpha z$.
Then $\tau$ generates  a cyclic subgroup $Z$ of order $q^n-1$ in $GL(n,q)$, known as a Singer cycle.

Let $\sigma$ be the Frobenius automorphism of $\F_{q^n}$ defined
by $\sigma(z)=z^q$ for all $z\in \F_{q^n}$. 
Then $\sigma$ acts $\F_q$-linearly on $\F_{q^n}$ and
we find that $\sigma\tau\sigma^{-1}=\tau^q$, so that $\sigma$ normalizes $Z$.

The  group generated by $\tau$ and $\sigma$  acts on $\F_{q^n}$.
We call this group the semilinear group of degree 1 and denote it by $\Gamma L(1,q^n)$. 
Thus $\Gamma L(1,q^n)$ is a subgroup
of $GL(n,q)$ of order $n(q^n-1)$, and 
$\Gamma L(1,q^n)$ is the semidirect product of a normal cyclic subgroup 
of order $q^n-1$ (the Singer cycle $Z$) with a cyclic group of order $n$. 
It is known that the normalizer of $Z$ in $GL(n,q)$ is $\Gamma L(1,q^n)$. Furthermore,
the determinant homomorphism from $\Gamma L(1,q^n)$ into $\F_{q}\setminus \{0\}$ is surjective
(see, for example, Satz 7.3, p.187, of \cite{Hup}). 
Thus, if we define
$$\Gamma L_1(1,q^n):=\Gamma L(1,q^n)\cap SL(n,q)$$
then $\Gamma L_1(1,q^n)$ has order $n(q^n-1)/(q-1)$. 

We study the action of  $\Gamma L_1(1,q^n)$ in the cases $q$ even and $q$ odd separately.
The case of $q$ odd is easier.

The following lemma shows that $\Gamma L_1(1,q^n)$ is not transitive in the cases under consideration here.

\begin{lemma} \label{no_transitive_action}
Let $n$ and $q$ both be odd. Then $\Gamma L_1(1,q^n)$ 
does not act transitively on $\F_{q^n}\setminus \{0\}$.
\end{lemma}

\begin{proof}
Suppose that $\Gamma L_1(1,q^n)$ acts transitively on $\F_{q^n}\setminus \{0\}$. 
It follows then that $q^n-1$ divides
$n(q^n-1)/(q-1)$ and hence $q-1$ divides $n$. This is impossible as $q-1$ is even and $n$ odd under the stated hypothesis.
\end{proof}

We need an improvement of Lemma \ref{no_transitive_action} to deal with the case when $q$ is a power of 2. 
We now restrict ourselves
to prime dimension $r$, as this is the focus of our paper.

\begin{lemma} \label{no_transitive_action_in_characteristic_2}
Let $r$ be an odd prime and let $q=2^m$.
The group $\Gamma L_1(1,q^r)$  acts transitively on $\F_{q^r}\setminus \{0\}$
if and only if $q=2$.
\end{lemma}

\begin{proof}
In the case that $q=2$, $\Gamma L_1(1,2^n)$  does act transitively on $\F_{2^n}\setminus \{0\}$ 
for all positive integers $n$,
because $GL(n,2)=SL(n,2)$ and thus $\Gamma L_1(1,2^n)=\Gamma L(1,2^n)$. 

Conversely, suppose that $\Gamma L_1(1,q^r)$ acts transitively on $\F_{q^r}\setminus \{0\}$. 
Then by the proof of Lemma \ref{no_transitive_action},
$q-1$ divides $r$. Thus, since $r$ is a prime, either $q-1=1$ or $q-1=r$. 

The proof when $q=2$ is complete.
Let us suppose that we are in the other case, where $q-1=r$.
It remains to find a contradiction to the hypothesis that
$\Gamma L_1(1,q^r)$  acts transitively on $\F_{q^r}\setminus \{0\}$. 
Since $q-1=r$  we have $|\Gamma L_1(1,q^r)|=q^r-1$. 
Furthermore, since we are assuming that the action of $\Gamma L_1(1,q^r)$ is transitive,
this action must be regular, because $|\F_{q^r}\setminus \{0\}|=q^r-1$.

The normal basis theorem for finite fields implies that $\F_{q^r}$ has an $\F_q$-basis consisting of elements $\beta$, $\sigma(\beta)$,
\dots, $\sigma^{r-1}(\beta)$ for some element $\beta$ of $\F_{q^r}$. Clearly, $\sigma$ permutes the basis vectors as an $r$-cycle.
It follows that $\det \sigma=1$, since we are working in characteristic 2, so $\sigma$ is in $SL(r,q)$
and hence in $\Gamma L_1(1,q^r)$.
However, $\sigma$ fixes the nonzero element 1 of $\F_{q^r}$, 
which contradicts the fact that $\Gamma L_1(1,q^r)$  acts regularly.
\end{proof}

We have used the finite field $\F_{q^r}$ as our vector space $\F_q^r$, because this makes 
the subgroups we are using easy to see. Since these two vector spaces are isomorphic, the same
results of course hold for $\F_q^r$.

\section{Proof of Main Theorem}

We now present the  classification of subgroups of 
$SL(r,q)$ that act transitively on $\F_q^r\setminus \{0\}$, which follows from results in the literature.

 \begin{thm}\label{result_from_FM} 
 Let $q$ be a prime power, let $r$ be an odd prime, and let $G$ be a subgroup of 
 $SL(r,q)$ that acts transitively on the nonzero vectors of $\mathbb{F}_q^r$.
 Then either $G=SL(r,q)$ or $G\leq \Gamma L(1,q^r)$.
 \end{thm}
 
 A proof of Theorem \ref{result_from_FM} can be found in chapter 3 of the book \cite{LPS}
by Liebeck et al.
The proof uses 
theorems of Hering \cite{He1}, \cite{He2}.
We note that Hering's work uses the classification of finite simple
groups. We also note that there is a  presentation given by Sambale, \cite{S}, Theorem 15.1, p.197, who provides some corrections and explanations to previously published
statements, but this discussion deals only with $q$ equal to a prime.
The statement of Theorem \ref{result_from_FM} 
for $q$ odd can also be derived from the statement of Proposition 3.2 in Ferraguti-Micheli \cite{FM}.

\begin{thm} \label{only_SL(r)_acts_transitively}
Let $r$ be an odd prime  and let $q>2$ be a power of a prime.  Then the only subgroup of 
$SL(r,q)$ that acts transitively on $\F_q^r\setminus \{0\}$ is $SL(r,q)$ itself.
\end{thm}
 
\begin{proof}
Let $G$ be a subgroup of $SL(r,q)$ that enjoys the stated transitivity property. 
By Theorem \ref{result_from_FM},  either $G=SL(r,q)$ or $G\leq \Gamma L_1 (1,q^r)$.
By Lemma  \ref{no_transitive_action} when $q$ is odd,
and Lemma  \ref{no_transitive_action_in_characteristic_2}  when $q$ is even, 
$G\leq \Gamma L_1 (1,q^r)$ is impossible.
\end{proof}

We can now deduce the promised theorem relating to the Galois group of an irreducible $q$-polynomial whose
$q$-degree is a prime.

\begin{thm} \label{Galois_group_is_SL(r,q)}
Let $F$ be a field of characteristic $p$ that contains $\F_q$, where $q$ is a power of $p$.
Let $L$ be a monic $q$-polynomial of $q$-degree $r$ in $F[x]$, where $r$ is an odd prime.  Suppose that the coefficient
of $x$ in $L$ is $-1$ and
$L(x)/x$ is irreducible over $F$. Then the Galois group of $L$ over $F$ is isomorphic to the special linear
group $SL(r,q)$, provided that $q>2$.\\ (More generally, the same conclusion holds if the coefficient of $x$ in $L$ is $-\mu^{q-1}$, where $\mu$ is any nonzero
element of $F$.)
\end{thm}

Theorem  \ref{Galois_group_is_SL(r,q)}
is immediate from Theorem  \ref{transitive_subgroup_of_SL}
and Theorem  \ref{only_SL(r)_acts_transitively}. In the next section, we deal with the $q=2$ version of Theorem
\ref{Galois_group_is_SL(r,q)}, which involves the appearance of two further subgroups of
$SL(r,2)$.

{\bf Example:} We consider an example to illustrate Theorem \ref{Galois_group_is_SL(r,q)}. 
Let $E$ be any field of characteristic $p$ that contains
$\F_q$. Let
$F$ be the function field $E(t)$, where $t$ is an indeterminate over $E$. Let $r$ be an odd prime and let $f(x)$ be a monic $q$-polynomial in $E[x]$ of $q$-degree $r$, with the coefficient of $x$ in $f$ equal to $-1$. Let $g(x)$ also be a nonzero $q$-polynomial in $E[x]$ whose $q$-degree is less than $r$, with
the coefficient of $x$ in $g$ equal to 0. Suppose also that $f(x)/x$ and $g(x)/x$ are relatively prime. Then the polynomial $f(x)+tg(x)$ in $F[x]$ is a monic
$q$-polynomial of $q$-degree $r$, 
the  coefficient of $x$  is $-1$,
and $(f(x)+tg(x))/x$ is irreducible in $F[x]$.   It follows that this polynomial has Galois group $SL(r,q)$ over $F$.

We remark that the hypotheses of Theorem  \ref{Galois_group_is_SL(r,q)}
can never be satisfied when $F$ is a finite field, as is easy to show directly.

\section{Analysis when $q=2$}

An amendment needs to be made to Theorem \ref{Galois_group_is_SL(r,q)} to describe what happens
when $q=2$. As we mentioned above, two further subgroups of $SL(r,2)$ related to $\Gamma L(1,2^r)$ must be taken into consideration. 

We recall that $GL(n,2)=SL(n,2)$ and thus $\Gamma L_1(1,2^n)=\Gamma L(1,2^n)$. 

Let $r$ be an odd prime. We know that the group $\Gamma L(1,2^r)$ acts transitively on the nonzero elements of an $r$-dimensional vector space over $\mathbb{F}_2$. The Singer cycle subgroup $Z$ of $\Gamma L(1,2^r)$ also acts transitively. A simple argument will show
that $Z$ is the only proper subgroup that has this transitive action.

\begin{lemma} \label{transitive_action_of_subgroup}
Let $r$ be an odd prime and let $H$ be a proper subgroup of $\Gamma L(1,2^r)$ that acts
transitively on the nonzero elements of an $r$-dimensional vector space over $\mathbb{F}_2$. 
Then $H$ is the Singer cycle $Z$ in $\Gamma L(1,2^r)$.
\end{lemma}

\begin{proof}
Recall that $\Gamma L(1,2^r)$ has order $r(2^r-1)$.
 As $H$ acts transitively,
$2^r-1$ divides $|H|$. Thus,
since $H$ is a proper subgroup of order divisible by $2^r-1$, and 
$r$ is prime, we must have $|H|=2^r-1$.

It is easy to show that $r$ is relatively prime to $2^r-1$. Thus, since
$Z$ is a normal subgroup of $\Gamma L(1,2^r)$ whose order is relatively prime
to its index in $\Gamma L(1,2^r)$, elementary group theory shows that $Z$ is the unique subgroup 
of $\Gamma L(1,2^r)$ of order $2^r-1$ (in the language of group theory, $Z$ is a normal Hall subgroup of $\Gamma L(1,2^r)$
and hence unique of its order). Thus $H=Z$, as claimed.
\end{proof}

We have reached  a position where we can investigate the analogue of Theorem \ref{Galois_group_is_SL(r,q)} when $q=2$.

\begin{lemma} \label{three_subgroups_act_transitively}
Let $r$ be an odd prime.  Then, up
to conjugacy in $SL(r,2)$, the only proper subgroups of 
$SL(r,2)$ that acts transitively on $\F_2^r\setminus \{0\}$ are $\Gamma L(1,2^r)$
and the Singer cycle subgroup $Z$ of $\Gamma L(1,2^r)$
.
\end{lemma}
 
\begin{proof}
Let $G$ be a subgroup of $SL(r,2)$ that enjoys the stated transitivity property. 
By Theorem \ref{result_from_FM},  either $G=SL(r,2)$ or $G\leq \Gamma L(1,q^r)$.
By Lemma  \ref{transitive_action_of_subgroup}, the only proper subgroup of $\Gamma L(1,q^r)$
that acts transitively is $Z$, and we thus have established the lemma.
\end{proof}

\begin{thm} \label{Galois_group_is_SL(r,2)}
Let $F$ be a field of characteristic $2$.
Let $L$ be a monic $2$-polynomial of $2$-degree $r$ in $F[x]$, where $r$ is an odd prime.  Suppose that the coefficient
of $x$ in $L$ is $1$ and
$L(x)/x$ is irreducible over $F$. Then the Galois group of $L$ over $F$ is isomorphic to 
one of the following:
\begin{enumerate}
\item the special linear
group $SL(r,2)$, 
\item  the group $\Gamma L(1,2^r)$ of order $r(2^r-1)$, 
\item  the Singer
cycle $Z$ of order $2^r-1$. 
\end{enumerate}
Each possibility may occur for suitable $F$.
\end{thm}

\begin{proof}
Let $G$ denote the Galois group of $L$ over $F$. The arguments of the previous sections show that
$G$ is isomorphic to a subgroup of $SL(r,2)$ that acts transitively on $\F_2^r\setminus \{0\}$.
Lemma \ref{three_subgroups_act_transitively} identifies the three possibilities for $G$
up to isomorphism.

We proceed to provide examples of fields $F$ for which the Galois group has the stated
isomorphism type. We first take $F=\mathbb{F}_2$. Let
\[
a(x)=\sum_{i=0}^r a_ix^i
\]
be a primitive polynomial of degree $r$ over $\mathbb{F}_2$
(the minimal polynomial of a generator of the multiplicative group $\mathbb{F}_{2^r}\setminus \{ 0 \}$).

Then the polynomial
\[
L(x)=\sum_{i=0}^r a_ix^{2^i}
\]
is called the linearized $2$-associate of $a(x)$. Since $a(x)$ is irreducible of order $2^r-1$,
Theorem 3.63 of \cite{LN} 
implies that $L(x)/x$ is irreducible over $\mathbb{F}_2$. Thus
the Galois group of the 2-polynomial $L(x)$ is cyclic of order $2^r-1$.

We next take $F=\mathbb{F}_2(t)$, where $t$ is transcendental over $\mathbb{F}_2$, to provide
the other two examples. We first take 
\[
L(x)=x^{2^r}+tx.
\]
Then we may easily prove that $L(x)/x$ is irreducible over $F$ and has Galois group $\Gamma L(1,2^r)$.

Finally we take
\[
L(x)=x^{2^r}+x^2+tx
\]
over $\mathbb{F}_2(t)$. A result of Abhyankar, \cite{Ab}, Claim 1.2, shows that
$L(x)$ has Galois group $SL(r,2)$.
\end{proof}

\section{Projective Polynomials}

There is another Galois group related to the polynomial $L$ described above, which we introduce now.  Let $M(x)=L(x)/x$. 
Then it is easy to see that $M(x)$ is a monic polynomial in $x^{q-1}$, say $M(x)=P(x^{q-1})$. The polynomial $P$ is monic 
of degree $(q^r-1)/(q-1)$ and is called the projective polynomial associated to $L$. Since we are assuming that $M$ is irreducible,
$P$ is also irreducible over $F$. 

\begin{cor} \label{Galois_group_is_PSL(r,q)}
Let $F$ be a field of characteristic $p$ that contains $\F_q$, where $q$ is a power of $p$. 
Let $L$ be a monic $q$-polynomial of $q$-degree $r$ in $F[x]$, where $r$ is an odd prime.  Suppose that $L(x)/x$ is irreducible over $F$ and the coefficient
of $x$ in $L$ is $-1$. Let $P(x)$ be the projective polynomial of degree $(q^r-1)/(q-1)$ associated to $L$.
Then the Galois group of $P$ over $F$ is isomorphic to the projective special linear
group $PSL(r,q)$.
\end{cor}

\begin{proof}
Given that $P(x^{q-1})=L(x)/x$, it follows that the Galois group of $P$ over $F$ is a homomorphic image
of that of $L$, since the splitting field of $L$ over $F$ contains that of $P$. Now we have shown that the Galois group of $L$ over $F$ is isomorphic to $SL(r,q)$. We consider
two cases. Suppose that $r$ does not divide $q-1$. Then $SL(r,q)$ is a nonabelian finite simple group and
$PSL(r,q)$ is isomorphic to $SL(r,q)$. The corollary is immediate in this case.

Suppose next that $r$ divides $q-1$. Then $\F_q^*$ contains an element $\omega$, say, of order $r$ and similarly
$SL(r,q)$ contains a central subgroup, $Z$, say, of order $r$ generated by the scalar matrix $\omega I$. The structure
theory of special linear groups implies that $Z$ is the unique proper normal subgroup of $SL(r,q)$. This follows since $PSL(r,q)=SL(r,q)/Z$
is a simple group (see, for example, Hauptsatz 6.13, p.182, of \cite{Hup}).
Since $\omega I$ acts trivially on the $(q-1)$-powers of the roots of $L$, and hence trivially on the roots of $P$, the homomorphism from
the Galois group of $L$ onto the Galois group of $P$ contains $Z$ in its kernel. The kernel is then precisely
$Z$ and hence the image is $PSL(r,q)$.
\end{proof}

\section{Polynomials of q-degree 2}

Our results up to now have been for $q$-polynomials of $q$-degree $r$, where $r$ is an odd prime.
We conclude the paper by considering $q$-polynomials of $q$-degree 2. 
As we mentioned earlier, we can employ a theorem
of Dickson to prove reasonably complete results. We begin the section by proving a lemma
that is useful for making certain constructions of Galois groups. The result itself must
be well known.

\begin{lemma} \label{irreducibility_criterion}
Let $F$ be a field and let $P$ be a monic irreducible separable polynomial in $F[x]$.
Let $E$ be a splitting field for $P$ over $F$ and let $G$ be the Galois group of $E$ over
$F$. Let $H$ be any subgroup of $G$ that acts transitively on the roots of $P$ in $E$
and let $E^H$ be the fixed field of $H$ in $E$. Then $P$ is irreducible over $E^H$ and $H$ is its Galois group.
\end{lemma}

\begin{proof}
It follows from Artin's theorem, \cite{L}, Theorem 1.8, p.264, that $E$ is a Galois extension
of $E^H$ with Galois group $H$. 

Now let $Q$ be a monic irreducible factor of $P$ in $E^H[x]$ and let $\alpha$ be a root of
$Q$ in $E$ (the roots of $Q$ lie in $E$ since $Q$ divides $P$). Then since $H$ fixes all elements
of $E^H$, $\sigma(\alpha)$ is also a root of $Q$ for any $\sigma$ in $H$. 
Thus, since $H$ acts transitively on the roots of $P$, $P$ and $Q$ have the same set of roots.
But as they are both monic polynomials, $P=Q$, and we have thus shown that $P$
is irreducible over $E^H$.
\end{proof}

We also use the following theorem in connection with the above-mentioned constructions.

\begin{thm} \label{Serre's_SL_2_theorem}
Let $q$ be a power of a prime and let $F=\mathbb{F}_q(t)$, where $t$ is transcendental
over $\mathbb{F}_q$. Then the polynomial $x^{q^2-1}+tx^{q-1}+1$ in $F[x]$ has Galois group
isomorphic to $SL(2,q)$. 
\end{thm}

\begin{proof}
A theorem of Serre, found in an appendix to  \cite{Ab}, p.131, shows that the Galois
group of $f(x)=x^{q+1}+tx+1$ over $F$ is isomorphic to the group $PSL(2,q)$.
Recall that $PSL(2,q)=SL(2,q)$ if $q$ is even, and $PSL(2,q)$ is the quotient of $SL(2,q)$ by its central
subgroup of order 2 if $q$ is odd.

The polynomial $L(x)=xf(x^{q-1})$ is a $q$-polynomial of $q$-degree 2 and thus its Galois
group $G$, say, is isomorphic to a subgroup of $SL(2,q)$ by Corollary \ref{determinants_are_1}. Now since
$f(x^{q-1})=L(x)/x$, the Galois group of $f(x)$ over $F$ is a homomorphic image of $G$, for
a splitting field of $L(x)/x$ over $F$ contains a splitting field for $f(x)$ over $F$.

We can rapidly dispose of the case when $q$ is even. For when $q$ is even, $SL(2,q)=PSL(2,q)$
and the homomorphism introduced above must be an isomorphism. Thus $G$ is isomorphic to
$SL(2,q)$.

We turn to the case when $q$ is odd, which requires a little more analysis. Since $|SL(2,q)|=q(q^2-1)$ and $|PSL(2,q)|=q(q^2-1)/2$, we have $|G|=|SL(2,q)|$ or $|G|=|SL(2,q)|/2$. Clearly,
the first possibility implies that $G$ is isomorphic to $SL(2,q)$. We must now show that
the second case cannot happen. 

For, suppose that $|G|=|SL(2,q)|/2=|PSL(2,q)|$. Then $G$ is isomorphic to $PSL(2,q)$. But $G$
is isomorphic to a subgroup of $SL(2,q)$, and $SL(2,q)$ contains a unique involution, namely
$-I$. Since $|G|$ is even, $G$ must also contain a unique involution. This is a contradiction,
since it is easy to see that $PSL(2,q)$ contains more than one involution. This excludes
the second possibility.
\end{proof} 

We turn now to consideration of the Galois groups attached to certain $q$-polynomials of
$q$-degree 2. The main result is complicated by the occurrence of a number of exceptional cases
which we feel are of special interest.

\begin{thm} \label{Galois_group_is_SL(2,q)}
Suppose that the field $F$ contains the field $\F_q$. Let $L(x)=x^{q^2}+bx^q+x$ in $F[x]$ is and suppose that $L(x)/x$ is irreducible over $F$. Then the Galois group $G$ of $L$
over $F$ is isomorphic to $SL(2,q)$ when $q$ is not equal to any of $2$, $3$, $5$, $7$, $11$.

In the exceptional cases, the following hold.

\begin{enumerate}

\item  When $q=2$, $G$ is isomorphic to a cyclic group of order $3$ or to
$SL(2,2)$.

\item  When $q=3$, $G$ is isomorphic to a quaternion group of order $8$ or to
$SL(2,3)$. 

\item  When $q=5$, $G$ is isomorphic to 
$SL(2,3)$ or to $SL(2,5)$.

\item  When $q=7$, $G$ is isomorphic to the binary octahedral group of order $48$ (a double
cover of the symmetric group of degree $4$ with generalized quaternion Sylow $2$-subgroup)
or to $SL(2,7)$.

\item  When $q=11$, $G$ is isomorphic to 
$SL(2,5)$ or to $SL(2,11)$.
\end{enumerate}
Each of the five exceptional cases can occur for a suitably chosen field $F$, dependent
on the group. Specifically, $F$ can be taken to be an extension field of degree $q$
of the function field $\mathbb{F}_q(t)$.
\end{thm}

\begin{proof}
 Corollary \ref{determinants_are_1} shows that $G$ is isomorphic to
a subgroup of $SL(2,q)$. Furthermore, as we are assuming that $L(x)/x$
is irreducible, it follows that $q^2-1$ divides $|G|$. 

We divide the rest of the proof into two cases. Suppose that $q$ is a power of $2$. Then $SL(2,q)$ is identical
with $PSL(2,q)$ and $G$ is thus isomorphic to a subgroup of $PSL(2,q)$ of order divisible by $q^2-1$. Thus since $PSL(2,q)$ has order
$q(q^2-1)$, $G$ has index a divisor of $q$. Theorem 262, p.286, of 
Dickson \cite{D} shows that $PSL(2,q)$ has no subgroup
of index less than $q+1$ provided $q$ is different from 2. However, when $q=2$, $PSL(2,2)$
has a subgroup of order 3 that has index 2 in the group, and this accounts for one exception.

Suppose next that $q$ is odd. It will be convenient to shorten the presentation by considering $G$
to be a subgroup of $SL(2,q)$. 
Then as we noted in the previous proof,
$SL(2,q)$ contains a unique element of order 2, namely $-I$, and since $|G|$ is even,
$G$ contains $-I$. Let $\pi$ be the natural homomorphism from $SL(2,q)$ onto
$PSL(2,q)$.  We have $|\pi(G)|=|G|/2$ since $G$ contains $-I$ and
$-I$ is in the kernel of $\pi$.  

A straightforward
argument shows that the index of $\pi(G)$ in $PSL(2,q)$ divides $q$, and we again appeal to Dickson's theorem, \cite{D},
to show that no such subgroup exists when $q$ is different from $3$, $5$, $7$ or $11$ 
(but such subgroups do exist in those four exceptional cases).

We finish by considering the four odd exceptional (prime) values of $q$ where $\pi(G)$
has index $q$ in $PSL(2,q)$. We start with $q=3$. $PSL(2,3)$ has order 12 and hence $\pi(G)$
has order 4. It follows that $G$ has order 8 and is a Sylow 2-subgroup of $SL(2,3)$.
$G$ is therefore isomorphic to a quaternion group of order 8.

Next, we examine what happens when $q=5$. The group $PSL(2,5)$ is isomorphic
to the alternating group $A_5$ and any subgroup of index 5 in $A_5$ is isomorphic to $A_4$.
Thus $\pi(G)$ is isomorphic to $A_4$ and $G$ has order 24. Now $A_4$ has a normal Sylow 2-subgroup of order 4 and hence $G$ has a normal Sylow 2-subgroup of order 8, which is also
a Sylow 2-subgroup of $SL(2,5)$. Thus $G$ is contained in the normalizer of a Sylow 2-subgroup
of $SL(2,5)$. Since the normalizer of a Sylow 2-subgroup of $SL(2,5)$ is known to be $SL(2,3)$,
we see that $G$ is isomorphic to $SL(2,3)$.

Suppose next that $q=7$. By Dickson's theorem, a subgroup of index 7 in $PSL(2,7)$
is isomorphic to the symmetric group $S_4$. The group $G$ therefore has order 48 and it contains a Sylow
2-subgroup of order 16 of $SL(2,7)$, which is generalized quaternion of order 16. Let $Z$ denote
the central subgroup of $G$ generated by $-I$. The subgroup $Z$ is contained in the commutator subgroup
$G'$ of $G$, since it is already contained in the commutator subgroup of the Sylow 2-subgroup.
Thus $Z$ is contained in the intersection of $G'$ with the centre of $G$. This implies
that $G$ is a proper central extension of $S_4$. An argument of Schur, \cite{Sch}, second
paragraph of p.165, shows that there are at most
two isomorphism classes for $G$. The general linear group $GL(2,3)$ is 
one such proper central extension but the group $G$ has a generalized quaternion Sylow 2-subgroup
and is not isomorphic to $GL(2,3)$. It is thus a different central extension of $S_4$
known as the binary octahedral group. We have thus determined $G$ up to isomorphism.
We note
that Schur calls these proper central extensions {\it Darstellungsgruppen}.

The last remaining case is that of $q=11$. Dickson's theorem shows that a subgroup
of index 11 in $PSL(2,11)$ is isomorphic to $PSL(2,5)$. If we invoke the arguments
of the previous paragraph regarding covering groups and use a theorem of Schur, \cite{Hup}, Satz 25.7, p.646, we find that $G$ is isomorphic to $SL(2,5)$.

We close by considering constructions of the exceptional groups as Galois groups.
We take $F=\mathbb{F}_q(t)$ and use the polynomial $x^{q^2-1}+tx^{q-1}+1$ over $F$. We have shown
in Theorem \ref{Serre's_SL_2_theorem} that the Galois group of this polynomial over $F$ is isomorphic to $SL(2,q)$. For the five exceptional values of $q$, $SL(2,q)$ contains a subgroup of index $q$ that acts transitively
on the underlying two-dimensional vector space over $\mathbb{F}_q$, and this subgroup
is isomorphic to the appropriate group listed above, for the same reasons as we provided
in the argument of the last few paragraphs. We now appeal to Lemma \ref{irreducibility_criterion}
to deduce that $F$ has an extension field $E$, say, with $|E:F|=q$, such that 
$x^{q^2-1}+tx^{q-1}+1$ is irreducible over $E$ and has Galois group isomorphic
to the relevant exceptional group associated to $q$.
\end{proof}

We remark that  three of the exceptional groups we have described above, those of order 24, 48 and 120,
are well known in group theory and geometry. They are called binary polyhedral groups, since each
is the double cover of the symmetry group of a regular polyhedron  (Platonic solid). They occur
as subgroups of the group of unit real quaternions and were studied by Hamilton.

\bigskip

{\bf Acknowledgement.}
We thank the anonymous referee for careful reading of our submission and helpful suggestions.

\end{document}